\documentclass[11p]{article}
\usepackage[margin=2.5cm]{geometry} 
\usepackage{graphicx} 
\usepackage[colorlinks=false]{hyperref} 
\usepackage{amssymb} 
\usepackage{amsfonts} 
\usepackage{amsmath,amscd} 
\usepackage{scalerel}
\usepackage{amsthm}  
\usepackage{mathtools} 
\usepackage{titlesec} 
\usepackage{csquotes} 
\usepackage[shortlabels]{enumitem} 
\usepackage{subfiles} 
\usepackage{tikz-cd} 
\usepackage{wrapfig} 
\usepackage{cutwin}
\usepackage{caption} 
\usepackage{tabularx} 
\usepackage{hyperref}
\hypersetup{
     colorlinks   = true,
     citecolor    = green
}

\usepackage{nomencl}

\makenomenclature

\newcommand{\Addresses}{{
  \bigskip
  \footnotesize

  \textsc{Department of Mathematical Sciences, Norwegian University of Science and Technology,\\ 7491 Trondheim, Norway.}\par\nopagebreak
  \textit{E-mail address}: \texttt{eirik.berge@ntnu.no}
}}

\setlist[itemize]{itemsep=0mm} 
\numberwithin{equation}{section} 
\titleformat*{\section}{\Large \scshape\center} 
\titleformat*{\subsection}{\fontsize{14}{14} \sffamily} 

\theoremstyle{plain}
\newtheorem{theorem}{Theorem}[section]
\newtheorem*{theorem*}{Theorem} 
\newtheorem{lemma}[theorem]{Lemma}
\newtheorem*{conjecture*}{Conjecture}
\newtheorem{proposition}[theorem]{Proposition}
\newtheorem{corollary}[theorem]{Corollary}

\theoremstyle{definition}
\newtheorem{definition}[theorem]{Definition}
\newtheorem{example}[theorem]{Example}

\theoremstyle{remark}
\newtheorem*{remark}{Remark}

\begin{document}
\pagenumbering{gobble}
\title{\Huge{Interpolation in Wavelet Spaces and the HRT-Conjecture}}
\author{Eirik Berge}
\date{}
\maketitle

\pagenumbering{arabic}

\begin{abstract}
We investigate the wavelet spaces $\mathcal{W}_{g}(\mathcal{H}_{\pi})\subset L^{2}(G)$ arising from square integrable representations $\pi:G \to \mathcal{U}(\mathcal{H}_{\pi})$ of a locally compact group $G$. We show that the wavelet spaces are rigid in the sense that non-trivial intersection between them imposes strong restrictions. Moreover, we use this to derive consequences for wavelet transforms related to convexity and functions of positive type. Motivated by the reproducing kernel Hilbert space structure of wavelet spaces we examine an interpolation problem. In the setting of time-frequency analysis, this problem turns out to be equivalent to the HRT-Conjecture. Finally, we consider the problem of whether all the wavelet spaces $\mathcal{W}_{g}(\mathcal{H}_{\pi})$ of a locally compact group $G$ collectively exhaust the ambient space $L^{2}(G)$. We show that the answer is affirmative for compact groups, while negative for the reduced Heisenberg group. 
\end{abstract}

\section{Introduction}

In recent years there have been several fruitful connections between time-frequency analysis and abstract notions in both representation theory \cite{enstad2019balian, grochenig2019balian, grochenig2018orthonormal} and non-commutative geometry \cite{austad2020heisenberg, jakobsen2018duality, luef2009projective, luef2018balian}. This is mutually beneficial: The abstract machinery can illuminate many results in time-frequency analysis. On the other hand, the concrete setting of time-frequency analysis provides a useful playground for testing general conjectures. Building on this viewpoint, we consider a generalization of the \textit{Gabor spaces} \[V_{g}(L^{2}(\mathbb{R}^{n})) \subset L^{2}(\mathbb{R}^{2n}),\] where $V_{g}{f}$ is the short-time Fourier transform (STFT) of $f \in L^{2}(\mathbb{R}^{n})$ with respect to a non-zero window function $g \in L^{2}(\mathbb{R}^{n})$. The Gabor spaces have appeared explicitly in the time-frequency literature several times, e.g. \cite{abreu2015measures, hutnikova2010alternative}, as well as being implicitly present in much of the literature concerning the STFT. We refer the reader to \cite[Proposition 3.4.1]{grochenig2001foundations} where the connection between a certain Gabor space and the Bargmann-Fock space in complex analysis is described. Despite their importance, it is only recently that some of the basic properties of Gabor spaces have been examined in \cite{luef2020wiener}. Our goal is to derive results that are of interest both in the general setting and in the case of Gabor spaces. \par
Let us briefly describe the general setup of the paper. Consider a square integrable representation $\pi:G \to \mathcal{U}(\mathcal{H}_{\pi})$ of a locally compact group $G$ on a Hilbert space $\mathcal{H}_{\pi}$. We investigate the \textit{wavelet spaces} \[\mathcal{W}_{g}(\mathcal{H}_{\pi}) \subset L^{2}(G), \qquad \mathcal{W}_{g}f(x) := \langle f, \pi(x)g \rangle,\] where $g \in \mathcal{H}_{\pi}$ is an admissible vector and $x \in G$. The classical Gabor space $V_{g}(L^{2}(\mathbb{R}^{n}))$ is up to a phase-factor the wavelet space corresponding to the Schr\"{o}dinger representation of the reduced Heisenberg group $\mathbb{H}_{r}^{n}$. Wavelet spaces have appeared in the theory of coorbit spaces \cite{feichtinger1988unified, feichtinger1989banach1, feichtinger1989banach2} and have been independently studied in \cite{grossmann1985transforms, cont_wavelet_transform, wong2002wavelet}. The following result illustrates the rigidity of wavelet spaces.

\begin{theorem}
\label{first_result_introduction}
Let $\pi:G \to \mathcal{U}(\mathcal{H}_{\pi})$ and $\rho:G \to \mathcal{U}(\mathcal{H}_{\rho})$ be two square integrable representations with admissible vectors $g \in \mathcal{H}_{\pi}$ and $h \in \mathcal{H}_{\rho}$. Assume that the corresponding wavelet spaces intersect non-trivially, that is, \[\mathcal{W}_{g}(\mathcal{H}_{\pi}) \cap \mathcal{W}_{h}(\mathcal{H}_{\rho}) \neq \{0\}.\] Then $\mathcal{W}_{g}(\mathcal{H}_{\pi}) = \mathcal{W}_{h}(\mathcal{H}_{\rho})$ and there exists a unitary intertwining operator $T:\mathcal{H}_{\pi} \to \mathcal{H}_{\rho}$ satisfying $T(g) = h$.
\end{theorem}

A special case of Theorem \ref{first_result_introduction} reduces to the result in \cite[Theorem 4.2]{cont_wavelet_transform}. There are also two other noteworthy consequences of Theorem \ref{first_result_introduction} related to functions of positive type and convexity.

\begin{corollary}
Let $\pi:G \to \mathcal{U}(\mathcal{H}_{\pi})$ and $\rho:G \to \mathcal{U}(\mathcal{H}_{\rho})$ be square integrable representations with admissible vectors $g \in \mathcal{H}_{\pi}$ and $h \in \mathcal{H}_{\rho}$, respectively. Then $\mathcal{W}_{g}g - \mathcal{W}_{h}h$ is never a non-zero function of positive type.
\end{corollary}

\begin{corollary}
Let $\pi:G \to \mathcal{U}(\mathcal{H}_{\pi})$ be a square integrable representation of a unimodular group $G$ with admissible vectors $g,g_1,g_2 \in \mathcal{H}_{\pi}$. Assume we can write $\mathcal{W}_{g}g$ as a convex combination \[\mathcal{W}_{g}g = t \cdot \mathcal{W}_{g_1}g_{1} + (1 - t) \cdot \mathcal{W}_{g_2}{g_2},\] for some $t \in [0,1]$. Then $t \in \{0,1\}$ and we either have $g = cg_1$ or $g = c g_2$ for some $c \in \mathbb{T}$.
\end{corollary}

It is well-known that any wavelet space carries the structure of a reproducing kernel Hilbert space. This allows us to consider an interpolation problem for the wavelet spaces as follows: Consider distinct points $\{x_1, \dots, x_m\} \subset G$ and possibly non-distinct scalars $\lambda_1, \dots, \lambda_m \in \mathbb{C}$. We investigate whether there exists a function $F \in \mathcal{W}_{g}(\mathcal{H}_{\pi})$ that \textit{interpolates} these points, that is, $F(x_i) = \lambda_i$ for all $i = 1, \dots, m$. When this problem is always solvable the wavelet space $\mathcal{W}_{g}(\mathcal{H}_{\pi})$ is called \textit{fully interpolating}. This is a notion that has been extensively investigated in the reproducing kernel Hilbert space literature, see \cite[Chapter 3]{paulsen2016introduction}. However, in the case of the wavelet spaces the interpolation problem is to our knowledge only briefly mentioned in \cite{grossmann1985transforms}. \par 
We show in Proposition \ref{no_abelian_or_compact_are_fully_interpolating} that no wavelet space corresponding to a compact or abelian group can be fully interpolating. In the Gabor case, the interpolation problem turns out to be equivalent to the HRT-Conjecture regarding independence of time-frequency shifts. We will review the HRT-Conjecture in Section \ref{sec: Relationship_With_the_HRT-Conjecture} and show how it relates to the interpolation problem in Proposition \ref{equvalence_between_fully_interpolating_and_HRT_conjecture}. The partial results obtained for the HRT-Conjecture in the literature gives concrete examples of wavelet spaces that are fully interpolating. On the other hand, the interpolation problem gives an alternative view of the HRT-Conjecture that allows the tools from reproducing kernel Hilbert space theory to be applied. \par 
A theme throughout the paper is to utilize the theory of reproducing kernel Hilbert spaces to deduce properties of wavelet spaces. As an illustration of this, we will give a short proof of the following folklore result showing that tensor products are naturally incorporated in our setting. 

\begin{proposition}
\label{tensor_product_introduction}
Let $\pi:G \to \mathcal{U}(\mathcal{H}_{\pi})$ and $\rho:H \to \mathcal{U}(\mathcal{H}_{\rho})$ be two square integrable representations with admissible vectors $g \in \mathcal{H}_{\pi}$ and $h \in \mathcal{H}_{\rho}$. There is an isomorphism of reproducing kernel Hilbert spaces \[\mathcal{W}_{g \otimes h}(\mathcal{H}_{\pi} \hat{\otimes} \mathcal{H}_{\rho}) \simeq \mathcal{W}_{g}(\mathcal{H}_{\pi}) \hat{\otimes} \mathcal{W}_{h}(\mathcal{H}_{\rho}).\]
\end{proposition}

Finally, we would like to mention a problem where we are only able to obtain partial results. For a square integrable representation $\pi:G \to \mathcal{U}(\mathcal{H}_{\pi})$ we let $\mathcal{A}_{\pi}$ denote the equivalence classes of admissible vectors in $\mathcal{H}_{\pi}$ modulo rotations by elements of $\mathbb{T}$. We let $\widehat{G}_{s}$ denote the equivalence classes of square integrable representations of $G$ and consider the possibly non-direct sum of vector spaces
\begin{equation*}
\bigoplus_{\pi \in \widehat{G}_s} \underset{g \in \mathcal{A}_{\pi}}{\textrm{span}}\big\{\mathcal{W}_{g}f \, : \, f \in \mathcal{H}_{\pi}\big\} \subset L^{2}(G).
\end{equation*} 
Is this sum dense in $L^{2}(G)$ when $\widehat{G}_s \neq \emptyset$? Phrased conceptually, we question whether the wavelet spaces are collectively large enough to approximate any square integrable function. We say that a locally compact group $G$ is \textit{wavelet complete} when
\begin{equation*}
\overline{\bigoplus_{\pi \in \widehat{G}_s} \underset{g \in \mathcal{A}_{\pi}}{\textrm{span}}\big\{\mathcal{W}_{g}f \, : \, f \in \mathcal{H}_{\pi}\big\}} = L^{2}(G).
\end{equation*} 
For compact groups the affirmative answer follows directly from Peter-Weyl theory. Since commutative locally compact groups $G$ only have $\widehat{G}_{s} \neq \emptyset$ whenever they are compact, the conjecture is primarily interesting for non-abelian groups. The following result shows that wavelet completeness is a non-trivial notion.

\begin{proposition}
The reduced Heisenberg groups $\mathbb{H}_{r}^{n}$ are not wavelet complete. 
\end{proposition}

The structure of the paper is as follows: In Section \ref{sec: Preliminaries} we review the nessesary material regarding square integrable representations and reproducing kernel Hilbert spaces. The examination of wavelet spaces starts in Section \ref{sec: Basic_Properties_of_the_Wavelet_Spaces} where we discuss basic properties. In Section \ref{sec: Disjointedness_of_Wavelet_Spaces} we show the disjointedness of the wavelet spaces and the resulting convexity consequence by utilizing abstract notions from the theory of functions of positive type. The interpolation problem for the wavelet spaces will be taken up in Section \ref{sec: Interpolation_Problem}. We present the connection between the interpolation problem and the HRT-Conjecture in Section \ref{sec: Relationship_With_the_HRT-Conjecture}. Finally, we examine wavelet completeness in Section \ref{sec: Wavelet_Completeness}. The author would like to thank Are Austad, Stine M. Berge, Franz Luef, Eirik Skrettingland, and Jordy Timo van Velthoven for valuable input.

\section{Preliminaries}
\label{sec: Preliminaries}

We will begin by reviewing the two settings of interest, namely square integrable representations of locally compact groups and reproducing kernel Hilbert spaces. This is done to fix notation and terminology, as well as to make the rest of the paper accessible to a broader audience. Background information for both topics can be found respectively in the books \cite{folland2016course, daubechies1992ten, dahlke2015harmonic, deitmar2014principles} and \cite{paulsen2016introduction, berlinet2011reproducing}.

\subsection{Square Integrable Representations}
\label{sec: Square_Integrable_Representations}

Let $G$ be a locally compact group, that is, a Hausdorff topological space that is also a group such that the multiplication map $(x,y) \mapsto xy$ and inversion map $x \mapsto x^{-1}$ are both continuous. The most important result when it comes to locally compact groups is the existence of a unique left-invariant Radon measure $\mu_{L}$ on $G$ called the (left) \textit{Haar measure} on $G$. Whenever there is any measure-theoretic construction on $G$ mentioned, it will always be with respect to the left Haar measure. In particular, the integrability spaces $L^{p}(G)$ for $1 \leq p \leq \infty$ consist of measurable functions $f:G \to \mathbb{C}$ such that \[\|f\|_{L^{p}(G)} := \left(\int_{G} |f(x)|^{p} \, d\mu_{L}(x)\right)^{\frac{1}{p}} < \infty.\]
Moreover, given $f,g \in L^{1}(G)$ the \textit{convolution} between $f$ and $g$ is given by \[(f *_G g)(x) := \int_{G}f(y)g(y^{-1}x) \, d\mu_{L}(y), \qquad x \in G.\]
We mention that the convolution product on $L^{1}(G)$ is commutative if and only if the group $G$ is abelian. \par
Analogously to the left Haar measure, there exists a right Haar measure $\mu_{R}$ on $G$ that is right-invariant. How much the two measures $\mu_{L}$ and $\mu_{R}$ deviate is captured in the \textit{modular function} $\Delta$ on $G$. Its precise definition \cite[Section 2.4]{folland2016course} need not concern us. However, it is worth knowing that $\mu_{L} = \mu_{R}$ precisely when $\Delta$ is identically one. In this case, we write $\mu:= \mu_{L} = \mu_{R}$ and say that $G$ is \textit{unimodular}. Unimodular groups are abundant as they include abelian groups, compact groups, and discrete groups. 

\begin{definition}
Let $\mathcal{U}(\mathcal{H}_{\pi})$ denote the unitary operators on the Hilbert space $\mathcal{H}_{\pi}$. A group homomorphism $\pi:G \to \mathcal{U}(\mathcal{H}_{\pi})$ of a locally compact group $G$ is said to be a \textit{unitary representation} if the function
\begin{equation*}
    \mathcal{W}_{g}f(x) := \langle f, \pi(x)g \rangle_{\mathcal{H}_{\pi}}
\end{equation*}
is continuous on $G$ for any fixed $f,g \in \mathcal{H}_{\pi}$. We refer to $\mathcal{W}_{g}f$ as the \textit{wavelet transform} of $f$ with respect to $g$.
\end{definition}

The terminology for the wavelet transform is motivated by the classical continuous wavelet transform in wavelet analysis, see e.g. \cite{daubechies1992ten}. It is clear that $\mathcal{W}_{g}f$ is a bounded function on $G$ since \[|\mathcal{W}_{g}f(x)| \leq \|f\|_{\mathcal{H}_{\pi}}\|\pi(x)g\|_{\mathcal{H}_{\pi}} = \|f\|_{\mathcal{H}_{\pi}}\|g\|_{\mathcal{H}_{\pi}}, \quad x \in G, \, f,g \in \mathcal{H}_{\pi}.\] We will often fix $g \in \mathcal{H}_{\pi}$ and consider the map $\mathcal{W}_{g}:\mathcal{H}_{\pi} \to C_{b}(G)$ given by $\mathcal{W}_{g}(f) := \mathcal{W}_{g}f$, where $C_{b}(G)$ denotes the continuous and bounded functions on $G$. The spaces of primary interest for us will be $\mathcal{W}_{g}(\mathcal{H}_{\pi})$ as $g$ varies. However, as it stands now the conditions are to loose to deduce nice properties of the spaces $\mathcal{W}_{g}(\mathcal{H}_{\pi})$. Firstly, we will require that the representation $\pi$ is \textit{irreducible}, that is, there does not exist any non-trivial closed subspaces $\mathcal{M} \subset \mathcal{H}_{\pi}$ such that $\pi(x)\eta \in \mathcal{M}$ for every $x \in G$ and $\eta \in \mathcal{M}$. The main tool when working with irreducible representations is Schur's lemma \cite[Chapter 3]{folland2016course}:

\begin{lemma}
\label{Schur's_lemma}
Let $\pi:G \to \mathcal{U}(\mathcal{H}_{\pi})$ be a unitary representation of a locally compact group $G$. Then $\pi$ is irreducible if and only if every bounded linear operator $T:\mathcal{H}_{\pi} \to \mathcal{H}_{\pi}$ satisfying $T \circ \pi(x) = \pi(x) \circ T$ for all $x \in G$ is in fact a constant multiple of the identity transform $Id_{\mathcal{H}_{\pi}}$.
\end{lemma}

Bounded linear operators $T:\mathcal{H}_{\pi} \to \mathcal{H}_{\pi}$ satisfying $T \circ \pi(x) = \pi(x) \circ T$ for all $x \in G$ are called \textit{intertwining operators}. The second requirement we need on $\pi$ is one of integrability.

\begin{definition}
Let $\pi:G \to \mathcal{U}(\mathcal{H}_{\pi})$ be an irreducible unitary representation of a locally compact group $G$. We say that a non-zero vector $g \in \mathcal{H}_{\pi}$ is \textit{square integrable} if $\mathcal{W}_{g}g \in L^{2}(G)$. Similarly, we say that $\pi$ is \textit{square integrable} if there exists a square integrable vector in $\mathcal{H}_{\pi}$. 
\end{definition}

If $g \in \mathcal{H}_{\pi}$ is square integrable, then it actually follows that $\mathcal{W}_{g}f \in L^{2}(G)$ for all $f \in \mathcal{H}_{\pi}$. Moreover, the irreducibility of $\pi$ implies with little effort that the map $\mathcal{W}_{g}:\mathcal{H}_{\pi} \to C_{b}(G)$ is one-to-one. An improvement of these remarks is the following result of M. Duflo and C. C. Moore \cite{duflo1976regular} showing that the map $f \mapsto \mathcal{W}_{g}f$ is essentially an isometry.

\begin{proposition}
\label{Duflo_Moore_theorem}
Let $\pi:G \to \mathcal{U}(\mathcal{H}_{\pi})$ be a square integrable representation. There exists a unique positive, densely defined operator $C_{\pi}:\textrm{dom}(C_{\pi}) \subset \mathcal{H}_{\pi} \to \mathcal{H}_{\pi}$ with a densely defined inverse such that 
\begin{itemize}
    \item A non-zero element $g \in \mathcal{H}_{\pi}$ is square integrable if and only if $g \in \textrm{dom}(C_{\pi})$.
    \item For $g_1, g_2 \in \textrm{dom}(C_{\pi})$ and $f_1,f_2 \in \mathcal{H}_{\pi}$ we have the orthogonality relation 
    \begin{equation}
    \label{wavelet_transform_orthogonality}
        \langle \mathcal{W}_{g_1}f_1, \mathcal{W}_{g_2} f_2 \rangle_{L^{2}(G)} = \langle f_1, f_2 \rangle_{\mathcal{H}_{\pi}} \overline{\langle C_{\pi}g_1, C_{\pi}g_2 \rangle}_{\mathcal{H}_{\pi}}.
    \end{equation}
    \item The operator $C_{\pi}$ is injective and satisfies the invariance relation
    \begin{equation*}
        \pi(x)C_{\pi} = \sqrt{\Delta(x)} C_{\pi} \pi(x),
    \end{equation*}
    for all $x \in G$ where $\Delta$ denotes the modular function on $G$.
\end{itemize}
The operator $C_{\pi}$ is called the \textit{Duflo-Moore operator}.
\end{proposition}

We can always normalize a square integrable vector $g \in \mathcal{H}_{\pi}$ such that $\|C_{\pi}g\|_{\mathcal{H}_{\pi}} = 1$. A square integrable vector $g \in \mathcal{H}_{\pi}$ satisfying $\|C_{\pi}g\|_{\mathcal{H}_{\pi}} = 1$ is said to be \textit{admissible}. This condition is mainly one of convenience, and we will primarily work with admissible vectors. When $G$ is a unimodular group, then any square integrable representation $\pi$ of $G$ satisfies $\textrm{dom}(C_{\pi}) = \mathcal{H}_{\pi}$ and $C_{\pi} = c_{\pi} \cdot Id_{\mathcal{H}_{\pi}}$ for some $c_{\pi} > 0$. In this case, any non-zero vector $g \in \mathcal{H}_{\pi}$ is square integrable and admissibility simply reads $\|g\|_{\mathcal{H}_{\pi}} = c_{\pi}^{-1}$.

\subsection{Reproducing Kernel Hilbert Spaces}
\label{sec: Reproducing_Kernel_Hilbert_Spaces}

A Hilbert space $\mathcal{H}$ consisting of functions $f:X \to \mathbb{C}$ on a set $X$ does not need to relate pointwise notions with the abstract Hilbert space structure. For instance, convergence of a sequence $f_n \to f$ in the norm on $\mathcal{H}$ does not need to imply pointwise convergence $f_{n}(x) \to f(x)$ for every $x \in X$. However, by imposing that the natural evaluation functionals $E_{x}(f) := f(x)$ for $f \in \mathcal{H}$ and fixed $x \in X$ are bounded one obtains a strong relation between pointwise notions and the Hilbert space structure. 

\begin{definition}
A \textit{reproducing kernel Hilbert space} is a Hilbert space $\mathcal{H}$ consisting of functions $f:X \to \mathbb{C}$ on a set $X$ such that, for each $x \in X$, the evaluation functionals \[E_{x}(f) := f(x), \qquad f \in \mathcal{H},\] are bounded. If the collection $\{E_{x}\}_{x \in X}$ is uniformly bounded in norm we refer to $\mathcal{H}$ as \textit{uniform}.
\end{definition}

Examples of well-known reproducing kernel Hilbert spaces are the Paley-Wiener spaces $PW_{[-A,A]}$ for $A > 0$ and the Hardy space $H^{2}(\mathbb{D})$. We refer the reader to \cite{paulsen2016introduction} for a detailed discussion of these examples, while \cite{berlinet2011reproducing} gives examples of reproducing kernel Hilbert spaces related to stochastic processes. \par
There exists for each $x \in X$ a function $k_{x} \in \mathcal{H}$ such that $E_{x}(f) = \langle f, k_x \rangle_{\mathcal{H}}$ for all $f \in \mathcal{H}$. We refer to $k_{x}$ as the \textit{point kernel} corresponding to $x \in X$. The function $K:X \times X \to \mathbb{C}$ given by \[K(x,y) := \langle k_y, k_x \rangle_{\mathcal{H}} = k_{y}(x)\] is called the \textit{reproducing kernel} of $\mathcal{H}$. If $f_n \to f$ in the norm on $\mathcal{H}$, then 
\begin{equation}
\label{pointwise_convergence}
    |f_{n}(x) - f(x)| = |\langle f_n - f, k_{x}\rangle | \leq \|f_n - f\|_{\mathcal{H}}\|k_x\|_{\mathcal{H}} = \|f_n - f\|_{\mathcal{H}}\|E_x\|_{\mathcal{H}^{*}} \to 0.
\end{equation}
There are two general properties of reproducing kernel Hilbert spaces we will need in the sequel:
\begin{itemize}
    \item \cite[Proposition 2.13]{paulsen2016introduction} The reproducing kernel $K$ of a reproducing kernel Hilbert space is a \textit{kernel function}: Given any finite set of points $\Omega := \{x_1, \dots, x_m\} \subset X$, the matrix 
    \begin{equation}
    \label{omega_matrices}
        K_{\Omega} := \{K(x_{i},x_{j})\}_{i,j = 1}^{m}
    \end{equation} is positive semi-definite, that is, the eigenvalues of $K_{\Omega}$ are all non-negative.
    \item \cite[Proposition 2.3 and Theorem 2.4]{paulsen2016introduction} The reproducing kernel uniquely determines the resulting reproducing kernel Hilbert space: If $\mathcal{H}_{1}$ and $\mathcal{H}_2$ are both reproducing kernel Hilbert spaces on a set $X$ with the same reproducing kernel $K$, then $\mathcal{H}_1 = \mathcal{H}_2$ and $\|\cdot\|_{\mathcal{H}_{1}} = \|\cdot\|_{\mathcal{H}_{2}}.$ Conversely, if two reproducing kernel Hilbert spaces $\mathcal{H}_1$ and $\mathcal{H}_2$ coincide with equal norms, then the reproducing kernels for the spaces $\mathcal{H}_1$ and $\mathcal{H}_2$ are equal.
\end{itemize}

\begin{remark}
The reader should be aware that there is little consensus in the literature regarding the terminology \textit{positive definite}: Some authors, e.g. \cite{paulsen2016introduction}, use the term positive definite for the case $K_{\Omega} \geq 0$, while the majority will use the term positive definite to indicate that $K_{\Omega} > 0$. Hence we adopt the terminology \textit{positive semi-definite} for $K_{\Omega} \geq 0$ and \textit{strictly positive definite} for $K_{\Omega} > 0$ to minimize the possibility for any confusion. 
\end{remark}

It is important to note that the matrices $K_{\Omega}$ in \eqref{omega_matrices} do not need to be invertible. If all the matrices $K_{\Omega}$ are strictly positive definite, then we refer to the reproducing kernel Hilbert space $\mathcal{H}$ as \textit{fully interpolating}. The reason for this terminology will be clear in Section \ref{sec: Interpolation_Problem}.

\section{Basic Properties of Wavelet Spaces}
\label{sec: Basic_Properties_of_the_Wavelet_Spaces}

In this section we will define wavelet spaces and give their basic properties. This will connect the two topics reviewed in Section \ref{sec: Preliminaries} as the wavelet spaces have a natural reproducing kernel Hilbert space structure.

\begin{definition}
Let $\pi:G \to \mathcal{U}(\mathcal{H}_{\pi})$ be a square integrable representation of a locally compact group $G$ and fix an admissible vector $g \in \mathcal{H}_{\pi}$. The space \[\mathcal{W}_{g}(\mathcal{H}_{\pi}) \subset L^{2}(G)\] is called the \textit{(generalized) wavelet space} corresponding to the representation $\pi$ and the admissible vector $g$.
\end{definition}

The terminology is again motivated by the continuous wavelet transform in classical wavelet analysis. Notice that the wavelet space $\mathcal{W}_{g}(\mathcal{H}_{\pi})$ is a Hilbert space since it is a closed subspace of $L^{2}(G)$. Moreover, the norm $\mathcal{W}_{g}(\mathcal{H}_{\pi})$ inherits from $L^{2}(G)$ can be written by using \eqref{wavelet_transform_orthogonality} as \[\|\mathcal{W}_{g}f\|_{L^{2}(G)} = \|f\|_{\mathcal{H}_{\pi}}, \qquad f \in \mathcal{H}_{\pi}.\] \par 
An important property of the wavelet transform is that $\mathcal{W}_{g}$ is a unitary intertwining operator between $\pi$ and the left-regular representation on the space $\mathcal{W}_{g}(\mathcal{H}_{\pi})$: Let $L_{x}$ denote the left translation on functions $F \in \mathcal{W}_{g}(\mathcal{H}_{\pi})$ by $x \in G$, that is, $L_{x}F(y) = F(x^{-1}y)$ for $y \in G$. Then
\begin{equation*}
    \mathcal{W}_{g}(\pi(y)f)(x) = \langle \pi(y)f, \pi(x)g \rangle = \langle f, \pi(y^{-1})\pi(x)g \rangle = \mathcal{W}_{g}(f)(y^{-1}x) = L_{y}\mathcal{W}_{g}(f)(x),
\end{equation*}
for $x,y \in G$ and $f \in \mathcal{H}_{\pi}$. This shows that the wavelet spaces are left-invariant subspaces of $L^{2}(G)$. 

\begin{example}
Consider the \textit{reduced Heisenberg group} $\mathbb{H}_{r}^{n} := \mathbb{R}^n \times \mathbb{R}^n \times \mathbb{T}$ with the product \[\Big(x,\omega,e^{2\pi i \tau}\Big) \cdot \left(x',\omega',e^{2\pi i \tau'}\right) := \left(x + x', \omega + \omega', e^{2\pi i (\tau + \tau')}e^{\pi i (x' \cdot \omega - x \cdot \omega')}\right),\] for $x,x',\omega,\omega' \in \mathbb{R}^n$ and $\tau,\tau' \in \mathbb{R}$. The group $\mathbb{H}_{r}^{n}$ is non-abelian and unimodular with Haar measure equal to the usual product measure on $\mathbb{R}^n \times \mathbb{R}^n \times \mathbb{T}$. The \textit{Schr\"{o}dinger representation} $\rho_{r}:\mathbb{H}_{r}^{n} \to \mathcal{U}(L^{2}(\mathbb{R}^n))$ is the irreducible unitary representation given by 
\begin{equation}
\label{Schrodinger_representation}
\rho_{r}\Big(x,\omega,e^{2\pi i \tau}\Big) := e^{2\pi i \tau}e^{\pi i x \cdot \omega}T_{x}M_{\omega}, \qquad \Big(x,\omega,e^{2\pi i \tau}\Big) \in \mathbb{H}_{r}^{n},    
\end{equation}
where $T_{x}$ and $M_{\omega}$ are the \textit{time-shift} and \textit{frequency-shift} operators on $L^{2}(\mathbb{R}^{n})$ given by
\begin{equation*}
    T_{x}f(y) := f(y-x), \qquad M_{\omega}f(y) := e^{2 \pi i y \cdot \omega}f(y), \quad x,\omega \in \mathbb{R}^{n}.
\end{equation*}
A straightforward computation shows that the $n$-dimensional Gaussian function $g_{n}(x) := e^{-\frac{\pi}{2} x^2}$ for $x \in \mathbb{R}^{n}$ is square integrable for the Schr\"{o}dinger representation. Hence the Duflo-Moore operator satisfies $C_{\pi} = c_{\pi} \cdot Id_{L^{2}(\mathbb{R}^{n})}$ for some $c_{\pi} > 0$ since $\mathbb{H}_{r}^{n}$ is unimodular. In fact, we have $c_{\pi} = 1$ due to \cite[Theorem 3.2.1]{grochenig2001foundations}. Thus any normalized function in $L^{2}(\mathbb{R}^{n})$ is admissible.\par 
It is common in time-frequency analysis to consider the \textit{short-time Fourier transform (STFT)} \[V_{g}f(x,\omega) := \int_{\mathbb{R}^n}f(t)\overline{g(t-x)}e^{-2 \pi i t \cdot \omega} \, dt,\]
for $(x,\omega) \in \mathbb{R}^{2n}$ and $f,g \in L^{2}(\mathbb{R}^{n})$. 
The STFT is related to the wavelet transform of the reduced Heisenberg group by the formula 
\begin{equation}
\label{relation_wavelet_transform_STFT}
    \mathcal{W}_{g}f\Big(x,\omega,e^{2 \pi i \tau}\Big) = e^{-2\pi i \tau}e^{\pi i x \cdot \omega} V_{g}f(x,\omega), \qquad \Big(x,\omega, e^{2 \pi i \tau}\Big) \in \mathbb{H}_{r}^{n}.
\end{equation}
The phase-factor $e^{-2\pi i \tau}e^{\pi i x \cdot \omega}$ in \eqref{relation_wavelet_transform_STFT} is often irrelevant. Hence we will for the most part consider the STFT and the \textit{Gabor spaces} \[V_{g}(L^{2}(\mathbb{R}^{n})) \subset L^{2}(\mathbb{R}^{2n}),\] for $g \in L^{2}(\mathbb{R}^{n})$ with $\|g\|_{L^{2}(\mathbb{R}^{n})} = 1.$
\end{example}

\subsection{Wavelet Spaces as Reproducing Kernel Hilbert Spaces}

The fact that the wavelet spaces have a reproducing kernel Hilbert space structure originally appeared in the influential paper \cite{grossmann1985transforms}. Since then, it has been used in both special cases \cite{abreu2015measures} and in the general setting \cite{romero2020dual}. We provide the statement and brief proof for completeness as our assumptions are slightly different than in \cite{grossmann1985transforms} and include minor additions.

\begin{proposition}
\label{spaces_are_RKHS}
Let $\pi:G \to \mathcal{U}(\mathcal{H}_{\pi})$ be a square integrable representation with admissible vector $g \in \mathcal{H}_{\pi}$. The wavelet space $\mathcal{W}_{g}(\mathcal{H}_{\pi})$ is a uniform reproducing kernel Hilbert space. The point kernel $k_{x}$ corresponding to $x \in G$ is the function $k_{x} = \mathcal{W}_{g}(\pi(x)g)$, while the reproducing kernel $K:G \times G \to \mathbb{C}$ is given by \[K(x,y) = \langle \pi(y)g, \pi(x)g \rangle = \mathcal{W}_{g}(\pi(y)g)(x), \qquad x,y \in G.\] If $f_n \to f$ in the norm on $\mathcal{H}_{\pi}$, then 
\begin{equation}
\label{uniform_convergence_equation}
\mathcal{W}_{g}f_{n}(x) \to \mathcal{W}_{g}f(x)    
\end{equation}
uniformly for all $x \in G$. Moreover, if $h \in \mathcal{H}_{\pi}$ is another admissible vector then $\Psi_{g,h}:\mathcal{W}_{g}(\mathcal{H}_{\pi}) \to \mathcal{W}_{h}(\mathcal{H}_{\pi})$ given by \begin{equation}
\label{map_between_wavelet_spaces}
\Psi_{g,h}\left(\mathcal{W}_{g}f\right) := \mathcal{W}_{h}f, \qquad f \in \mathcal{H}_{\pi},    
\end{equation}
is an isomorphism of Hilbert spaces.
\end{proposition}

\begin{proof}
For $F \in \mathcal{W}_{g}(\mathcal{H}_{\pi})$ we have that $F(x) = \mathcal{W}_{g}\left(\mathcal{W}_{g}^{*}F\right)(x)$ since $\mathcal{W}_{g}$ is an isometry. Hence \[F(x) = \mathcal{W}_{g}\left(\mathcal{W}_{g}^{*}F\right)(x) = \left\langle \mathcal{W}_{g}^{*}F, \pi(x)g \right\rangle = \left\langle F, \mathcal{W}_{g}\left(\pi(x)g\right) \right\rangle.\] Since $k_{x} := \mathcal{W}_{g}\left(\pi(x)g\right) \in \mathcal{W}_{g}(\mathcal{H}_{\pi})$ the wavelet space $\mathcal{W}_{g}(\mathcal{H}_{\pi})$ is a reproducing kernel Hilbert space. The reproducing kernel $K$ can be written by using the orthogonality relations \eqref{wavelet_transform_orthogonality} as \[K(x,y) = \langle k_y, k_x \rangle = \left\langle \mathcal{W}_{g}\left(\pi(y)g\right), \mathcal{W}_{g}\left(\pi(x)g\right) \right\rangle = \langle \pi(y)g, \pi(x)g \rangle.\] If $E_{x}$ is the evaluation functional at the point $x \in G$ then  \[\|E_{x}\| = \|k_x\| = \|\mathcal{W}_{g}\left(\pi(x)g\right)\| = \|\pi(x)g\| = \|g\|.\]
Thus $\mathcal{W}_{g}(\mathcal{H}_{\pi})$ is uniform since the admissible vector $g \in \mathcal{H}_{\pi}$ is fixed. The computation \eqref{pointwise_convergence} shows that the convergence in \eqref{uniform_convergence_equation} is uniform. The map $\Psi_{g,h}$ is an isometry since \[\|\mathcal{W}_{h}f\|_{\mathcal{W}_{h}(\mathcal{H}_{\pi})} = \|f\|_{\mathcal{H}_{\pi}} = \|\mathcal{W}_{g}f\|_{\mathcal{W}_{g}(\mathcal{H}_{\pi})},\]
for all $f \in \mathcal{H}_{\pi}$. Finally, $\Psi_{g,h}$ is surjective as every element in $\mathcal{W}_{h}(\mathcal{H}_{\pi})$ is on the form $\mathcal{W}_{h}f$ for some $f \in \mathcal{H}_{\pi}$.
\end{proof}

\begin{remark}
The fact that the map $\Psi_{g,h}$ in \eqref{map_between_wavelet_spaces} is an isomorphism shows that the wavelet spaces corresponding to different admissible vectors can not be too different, e.g. their dimensions coincide. However, the wavelet spaces are still different as reproducing kernel Hilbert spaces since the map $\Psi_{g,h}$ does not in general preserve the reproducing kernels. 
\end{remark}

The wavelet transform $\mathcal{W}_{g}:\mathcal{H}_{\pi} \to L^{2}(G)$ is an isometry when $g \in \mathcal{H}_{\pi}$ is an admissible vector. Hence the projection from $L^{2}(G)$ to $\mathcal{W}_{g}(\mathcal{H}_{\pi})$ is given by $\mathcal{W}_{g} \circ \mathcal{W}_{g}^{*}$. A classical result in coorbit theory \cite{feichtinger1989banach1} known as the \textit{reproducing formula} describes this projection in terms of convolutions: The orthogonal projection from $L^{2}(G)$ to $\mathcal{W}_{g}(\mathcal{H}_{\pi})$ is explicitly given by 
\begin{equation*}
    \mathcal{W}_{g} \circ \mathcal{W}_{g}^{*}(F) = F *_{G} k_{e}, \qquad F \in L^{2}(G),
\end{equation*}
where $k_{e}(x) = \mathcal{W}_{g}g(x)$ is the point kernel corresponding to the identity element $e \in G$. The following basic result shows that the wavelet spaces automatically exhibit integrability properties that are not shared by general subspaces of $L^{2}(G)$.

\begin{proposition}
Let $\pi:G \to \mathcal{U}(\mathcal{H}_{\pi})$ be a square integrable representation and fix an admissible vector $g \in \mathcal{H}_{\pi}$. The wavelet space $\mathcal{W}_{g}(\mathcal{H}_{\pi})$ is continuously embedded into $L^{p}(G)$ for all $p \in [2,\infty]$. However, the wavelet space $\mathcal{W}_{g}(\mathcal{H}_{\pi})$ is not in general contianed in $L^{1}(G)$.
\end{proposition}

\begin{proof}
Notice that $\mathcal{W}_{g}(\mathcal{H}_{\pi})$ is continuously embedded in both $L^{2}(G)$ and $L^{\infty}(G)$: The first claim is obvious, while the second follows from the computation \[\|F\|_{L^{\infty}(G)} = \sup_{x \in G}|\langle k_x, F \rangle | \leq \sup_{x \in G} \|k_x\|_{\mathcal{W}_{g}(\mathcal{H}_{\pi})} \|F\|_{\mathcal{W}_{g}(\mathcal{H}_{\pi})} = \|g\|_{\mathcal{H}_{\pi}}\|F\|_{\mathcal{W}_{g}(\mathcal{H}_{\pi})},\]
for $F \in \mathcal{W}_{g}(\mathcal{H}_{\pi}).$
This observation implies that $\mathcal{W}_{g}(\mathcal{H}_{\pi})$ is continuously embedded into the intermediate spaces $L^{p}(G)$ for $p \in (2,\infty)$ as well since
    \[\|F\|_{L^{p}(G)} = \left(\int_{G}|F(x)|^{p-2}|F(x)|^2 \, d\mu_{L}(x)\right)^{\frac{1}{p}} \leq \|F\|_{L^{\infty}(G)}^{\frac{p-2}{p}}\|F\|_{\mathcal{W}_{g}(\mathcal{H}_{\pi})}^{\frac{2}{p}} \leq \|g\|_{\mathcal{H}_{\pi}}^{\frac{p-2}{p}}\|F\|_{\mathcal{W}_{g}(\mathcal{H}_{\pi})}.\]
Counterexamples to the last statement can be found in the time-frequency setting since the STFT satisfies $V_{g}g \in L^{1}(\mathbb{R}^{2n})$ only when $g$ is a continuous function on $\mathbb{R}^{n}$ by \cite[Proposition 12.1.4]{grochenig2001foundations}.
\end{proof}

Throughout the paper, we aim to emphasize how the reproducing kernel Hilbert space structure of the wavelet spaces is paramount. As a first example, we have the following existence result regarding generalized frame systems on reasonably general groups.

\begin{proposition}
\label{completeness_proposition}
Let $\pi:G \to \mathcal{U}(\mathcal{H}_{\pi})$ be a square integrable representation of a second countable locally compact group $G$ and fix an admissible vector $g \in \mathcal{H}_{\pi}$. There exists a countable set $\Lambda \subset G$ such that the generalized frame system $\{\pi(\lambda)g\}_{\lambda \in \Lambda}$ is complete in $\mathcal{H}_{\pi}$.
\end{proposition}

\begin{proof}
The second countability of $G$ is by \cite[Theorem 2]{de1978local} equivalent to the requirement that $L^{2}(G)$ is separable. Whence the subspace $\mathcal{W}_{g}(\mathcal{H}_{\pi}) \subset L^{2}(G)$ is also separable. By \cite[Lemma 11]{berlinet2011reproducing} there exists a countable set $\Lambda \subset G$ such that the collection of point kernels $k_{\lambda} = \mathcal{W}_{g}(\pi(\lambda)g)$ for $\lambda \in \Lambda$ is dense in $\mathcal{W}_{g}(\mathcal{H}_{\pi})$. Hence for $f \in \mathcal{H}_{\pi}$ the criterion \[\langle \mathcal{W}_{g}f, \mathcal{W}_{g}(\pi(\lambda)g) \rangle = 0\] for all $\lambda \in \Lambda$ forces $\mathcal{W}_{g}f \equiv 0$. The orthogonality relations  \eqref{wavelet_transform_orthogonality} and the injectivity of the Duflo-Moore operator implies that $\langle f, \pi(\lambda)g \rangle = 0$ for all $\lambda \in \Lambda$ only when $f = 0$. 
\end{proof}

\begin{remark}
The second countability condition in Proposition \ref{completeness_proposition} is only a sufficient requirement. In the proof of Proposition \ref{completeness_proposition} we need that the wavelet spaces $\mathcal{W}_{g}(\mathcal{H}_{\pi})$ are separable. This can happen when the ambient space $L^{2}(G)$ is not separable. In particular, the conclusion of Proposition \ref{completeness_proposition} holds for all square integrable representations corresponding to compact groups since the wavelet spaces are then finite-dimensional by \cite[Theorem 5.2]{folland2016course}.
\end{remark}

\subsection{Tensor Product of Wavelet Spaces}

Our setting involves both square integrable representations of locally compact groups as well as reproducing kernel Hilbert spaces. Both of these categories have a natural notion of a tensor product. We will use reproducing kernel Hilbert space arguments to show that these operations are compatible. Let us first briefly recall the different notions or tensor products involved. \par 
Consider two reproducing kernel Hilbert spaces $\mathcal{H}_i$ of functions on sets $X_i$ with reproducing kernels $K_{i}:X_{i} \times X_{i} \to \mathbb{C}$ for $i = 1,2$. We can form the tensor product $\mathcal{H}_{1} \otimes \mathcal{H}_2$ of Hilbert spaces in the usual way by requiring that \[\langle f_1 \otimes f_2, g_1 \otimes g_2 \rangle_{\mathcal{H}_{1} \otimes \mathcal{H}_2} := \langle f_1, g_1 \rangle_{\mathcal{H}_{1}}\langle f_2, g_2 \rangle_{\mathcal{H}_{2}},\]
where $f_1, g_1 \in \mathcal{H}_{1}$ and $f_2, g_2 \in \mathcal{H}_2$. This extends to an inner-product on $\mathcal{H}_{1} \otimes \mathcal{H}_2$ that is not in general complete. The completion of $\mathcal{H}_{1} \otimes \mathcal{H}_2$ with this inner-product is denoted by $\mathcal{H}_{1} \hat{\otimes} \mathcal{H}_2$ and called the \textit{tensor product} of the Hilbert spaces $\mathcal{H}_1$ and $\mathcal{H}_2$. Not surprisingly, the tensor product $\mathcal{H}_1 \hat{\otimes} \mathcal{H}_2$ can be identified with a reproducing kernel Hilbert space on the set $X_1 \times X_2$ as follows: Any element $u = \sum_{i = 1}^{n}f_i \otimes g_i \in \mathcal{H}_{1} \otimes \mathcal{H}_2$ can be identified with the function on $X_1 \times X_2$ given by $\tilde{u}(x,y) := \sum_{i = 1}^{n}f_i(x)g_i(y)$. This association extends to the completion $\mathcal{H}_{1} \hat{\otimes} \mathcal{H}_2$ and gives a well-defined linear isometry between $\mathcal{H}_{1} \hat{\otimes} \mathcal{H}_2$ and the reproducing kernel Hilbert space on $X_1 \times X_2$ with reproducing kernel \[K((x_1,y_1),(x_2,y_2)) := K_{1}(x_1,x_2)K_{2}(y_1,y_2), \quad x_1,x_2 \in X_1, \, \, y_1,y_2 \in X_2.\] \par
In the setting of unitary representations of locally compact groups we also have a notion of a tensor product. Consider two unitary representations $\pi:G \to \mathcal{U}(\mathcal{H}_{\pi})$ and $\rho:H \to \mathcal{U}(\mathcal{H}_{\rho})$ where $G$ and $H$ are locally compact groups. We can consider the \textit{tensor product representation} $\pi \otimes \rho$ given on elementary tensors $f_1 \otimes f_2$ by \[(\pi \otimes \rho)(x,y)(f_1 \otimes f_2) := \pi(x)f_1 \otimes \rho(y)f_2,\] for $x \in G$, $y \in H$, $f_1 \in \mathcal{H}_{\pi}$, and $f_2 \in \mathcal{H}_{\rho}$. This extends to arbitrary elements in $\mathcal{H}_{\pi} \hat{\otimes} \mathcal{H}_{\rho}$ and hence defines a unitary representation $\pi \otimes \rho:G \times H \to \mathcal{U}(\mathcal{H}_{\pi} \hat{\otimes} \mathcal{H}_{\rho})$. The following result shows that the two tensor product constructions we have described are compatible in a natural way.

\begin{proposition}
\label{tensor_product_result}
Let $\pi:G \to \mathcal{U}(\mathcal{H}_{\pi})$ and $\rho:H \to \mathcal{U}(\mathcal{H}_{\rho})$ be two square integrable representations with admissible vectors $g \in \mathcal{H}_{\pi}$ and $h \in \mathcal{H}_{\rho}$. There is an isomorphism of reproducing kernel Hilbert spaces \[\mathcal{W}_{g \otimes h}(\mathcal{H}_{\pi} \hat{\otimes} \mathcal{H}_{\rho}) \simeq \mathcal{W}_{g}(\mathcal{H}_{\pi}) \hat{\otimes} \mathcal{W}_{h}(\mathcal{H}_{\rho}).\]
\end{proposition}

\begin{proof}
Let us first check that the representation $\pi \otimes \rho$ is square integrable and that $g \otimes h \in \mathcal{H}_{\pi} \hat{\otimes} \mathcal{H}_{\rho}$ is admissible. For $x \in G$ and $y \in H$ we have \begin{align*} \left\|\mathcal{W}_{g \otimes h}g \otimes h\right\|_{L^{2}(G \times H)}^{2} & = \int_{G \times H} |\langle g \otimes h, (\pi \otimes \rho)(x,y)(g \otimes h) \rangle|^2 \, d\mu_{L}^{G \times H}(x,y) \\ & = \int_{G \times H} |\langle g, \pi(x)g \rangle \langle h, \rho(y)h \rangle|^2 \, d\mu_{L}^{G \times H}(x,y) \\ & = \int_{G} |\langle g, \pi(x)g \rangle|^2 \, d\mu_{L}^{G}(x) \int_{H} |\langle h, \rho(y)h \rangle|^2 \, d\mu_{L}^{H}(y) \\ & = \|\mathcal{W}_{g}g\|_{L^{2}(G)}^2 \|\mathcal{W}_{h}h\|_{L^{2}(H)}^2 < \infty.
\end{align*}
Moreover, we see from the above computation that we have the pointwise equality \[\mathcal{W}_{g \otimes h} g \otimes h(x,y) = \mathcal{W}_{g}g(x) \mathcal{W}_{h}h(y), \qquad x \in G, \, y \in H,\] as functions on $G \times H$. Since the reproducing kernels for the space $\mathcal{W}_{g \otimes h}(\mathcal{H}_{\pi} \hat{\otimes} \mathcal{H}_{\rho})$ and the space $\mathcal{W}_{g}(\mathcal{H}_{\pi}) \hat{\otimes} \mathcal{W}_{h}(\mathcal{H}_{\rho})$ coincide, the result follows from the uniqueness of reproducing kernels given in Section \ref{sec: Reproducing_Kernel_Hilbert_Spaces}.
\end{proof}

\begin{example}
Consider the Gabor space $V_{g_{n}}(L^{2}(\mathbb{R}^n))$ where $g_{n}(x) = e^{-\frac{\pi}{2}x^2}$ is the $n$-dimensional Gaussian function. Then Theorem \ref{tensor_product_result} implies that \[V_{g_{n}}\left(L^{2}(\mathbb{R}^n)\right) \hat{\otimes}\, V_{g_{n}}\left(L^{2}(\mathbb{R}^n)\right) \simeq V_{g_{n} \otimes g_{n}}\left(L^{2}(\mathbb{R}^n) \otimes L^{2}(\mathbb{R}^n)\right) \simeq V_{g_{2n}}\left(L^{2}(\mathbb{R}^{2n})\right),\] where $g_{2n}$ is the $2n$-dimensional Gaussian function. Although this is folklore knowledge, we emphasize the the simplicity of its derivation from the theory of reproducing kernel Hilbert spaces. Many function spaces in complex analysis, e.g. the Hardy spaces and Bergman spaces, satisfy similar tensorization rules \cite[Proposition 5.13 and Proposition 5.14]{paulsen2016introduction}. This is maybe not so surprising given the connection between the Gabor space $V_{g_n}(L^{2}(\mathbb{R}^{n}))$ and complex analysis given in \cite[Proposition 3.4.1]{grochenig2001foundations}.
\end{example}

\section{Rigidity of Wavelet Spaces}
\label{sec: Disjointedness_of_Wavelet_Spaces}

In this section we will investigate how wavelet spaces associated with (potentially) different representations are related. The main result in Theorem \ref{equality_of_wavelet_spaces_general_statement} have several noteworthy consequences. The first consequence in Corollary \ref{equality_of_wavelet_spaces_special_case} is a new proof of one of the main results in \cite[Theorem 4.2]{cont_wavelet_transform}. The other consequences, Corollary \ref{original_consequence_1} and Corollary \ref{convex_combination_corollary}, are new and illustrate the broad utility of Theorem \ref{equality_of_wavelet_spaces_general_statement}. Let us first consider an example of the general setting where things are greatly simplified.

\begin{example}
\label{Example_abelian_groups}
Let $G$ be a locally compact group that is abelian and let $\pi:G \to \mathcal{U}(\mathcal{H}_{\pi})$ be a square integrable representation. It follows from Schur's Lemma \ref{Schur's_lemma} that $\mathcal{H}_{\pi} \simeq \mathbb{C}$ and $\mathcal{U}(\mathcal{H}_{\pi}) \simeq \mathbb{T}$. We make these identifications and view $\pi$ as a map from $G$ to $\mathbb{T}$. What requirements do the square integrability impose? For $z \in \mathbb{C} \setminus \{0\}$ we have that \[\int_{G}|\langle z, \pi(x)z \rangle|^2 \, d\mu(x) = |z|^4\mu(G).\] Hence $\pi$ is square integrable if and only if $\mu(G) < \infty$. This is the case precisely when $G$ is compact. Since $G$ is unimodular, it follows from Proposition \ref{Duflo_Moore_theorem} that the Duflo-Moore operator $C_{\pi}$ is a positive constant multiple of the identity. That the constant is equal to one can be seen by direct verification, or by an application of Peter-Weyl theory \cite[Example 12.2.7]{deitmar2014principles}. Hence a complex number $z \in \mathbb{C}$ is admissible if and only if $z \in \mathbb{T}$. The wavelet spaces $\mathcal{W}_{z}(\mathbb{C})$ for $z \in \mathbb{T}$ are one-dimensional subspaces of $L^{2}(G)$ that are spanned by the elements $\mathcal{W}_{z}z$. Moreover, all the wavelet spaces $\mathcal{W}_{z}(\mathbb{C})$ coincide since $\mathcal{W}_{z}z = \mathcal{W}_{1}1$ for all $z \in \mathbb{T}$.
\end{example}

Notice that everything said in Example \ref{Example_abelian_groups} is independent of the representation in question: In the abelian case, all the wavelet spaces coincide even when we have two different representations $\pi:G \to \mathbb{T}$ and $\rho:G \to \mathbb{T}$. On the other hand, we always have that any two admissible vectors $z,w \in \mathbb{T}$ (regardless of the choice of representations) are related by $z = cw$ for some $c \in \mathbb{T}$. These elementary remarks motivate the following general result.

\begin{theorem}
\label{equality_of_wavelet_spaces_general_statement}
Let $\pi:G \to \mathcal{U}(\mathcal{H}_{\pi})$ and $\rho:G \to \mathcal{U}(\mathcal{H}_{\rho})$ be two square integrable representations with admissible vectors $g \in \mathcal{H}_{\pi}$ and $h \in \mathcal{H}_{\rho}$. Assume that the corresponding wavelet spaces intersect non-trivially, that is, \[\mathcal{W}_{g}(\mathcal{H}_{\pi}) \cap \mathcal{W}_{h}(\mathcal{H}_{\rho}) \neq \{0\}.\] Then $\mathcal{W}_{g}(\mathcal{H}_{\pi}) = \mathcal{W}_{h}(\mathcal{H}_{\rho})$ and there exists a unitary intertwining operator $T:\mathcal{H}_{\pi} \to \mathcal{H}_{\rho}$ satisfying $T(g) = h$.
\end{theorem}

\begin{proof}
Notice that the subspace $\mathcal{W}_{h}(\mathcal{H}_{\rho}) \cap \mathcal{W}_{g}(\mathcal{H}_{\pi}) \subset \mathcal{W}_{g}(\mathcal{H}_{\pi})$ is invariant under translations. Since $\pi$ is irreducible and $\mathcal{W}_{g}:\mathcal{H}_{\pi} \to \mathcal{W}_{g}(\mathcal{H}_{\pi})$ is a unitary intertwiner we have that $\mathcal{W}_{h}(\mathcal{H}_{\rho}) = \mathcal{W}_{g}(\mathcal{H}_{\pi})$.
The norms on $\mathcal{W}_{g}(\mathcal{H}_{\pi}) = \mathcal{W}_{h}(\mathcal{H}_{\rho})$ both coincide with the restriction of the $L^{2}(G)$-norm. Hence $\mathcal{W}_{g}(\mathcal{H}_{\pi})$ and $\mathcal{W}_{h}(\mathcal{H}_{\rho})$ are reproducing kernel Hilbert spaces that coincide with equal norms. By the uniqueness statements given in Section \ref{sec: Reproducing_Kernel_Hilbert_Spaces} the two reproducing kernels coincide \[\mathcal{W}_{g}(\pi(y)g)(x) = \mathcal{W}_{h}(\rho(y)h)(x), \qquad x,y \in G.\]
Since $\mathcal{W}_{g}(\pi(y)g)(x) = L_{y}\mathcal{W}_{g}g(x)$ and $\mathcal{W}_{h}(\rho(y)h)(x) = L_{y}\mathcal{W}_{h}h(x)$, all the information we need is contained in the equality 
\begin{equation}
\label{equality_of_kernels}
    \mathcal{W}_{g}g(x) = \mathcal{W}_{h}h(x), \qquad x \in G.
\end{equation}
\par 
To define the map $T:\mathcal{H}_{\pi} \to \mathcal{H}_{\rho}$ we first require that $T(g) = h$. Moreover, for $T$ to be an intertwining operator, we need that \[T(\pi(x)g) = \rho(x)h, \qquad x \in G.\] Since $\pi$ is irreducible the set $\mathcal{M}_{g} := \textrm{span}\{\pi(x)g\}_{x \in G}$ is dense in $\mathcal{H}_{\pi}$. To see that $T$ extends to all of $\mathcal{H}_{\pi}$ we will show that it is an isometry on the subspace $\mathcal{M}_{g}$: For $x,y \in G$ we have \[\langle T(\pi(x)g), T(\pi(y)g) \rangle _{\mathcal{H}_{\pi}} = \langle \rho(x)h, \rho(y)h \rangle _{\mathcal{H}_{\pi}} = \langle h, \rho(x^{-1}y)h \rangle _{\mathcal{H}_{\pi}} = \mathcal{W}_{h}h(x^{-1}y).\] Hence we obtain from \eqref{equality_of_kernels} that \[\langle T(\pi(x)g), T(\pi(y)g) \rangle _{\mathcal{H}_{\pi}} = \mathcal{W}_{g}g(x^{-1}y) = \langle \pi(x)g, \pi(y)g \rangle _{\mathcal{H}_{\pi}}.\] The map $T$ is surjective since $\textrm{span}\{\rho(x)h\}_{x \in G}$ is dense in $\mathcal{H}_{\rho}$ due to the irreducibility of $\rho$. Hence $T$ is a unitary map. For $g \in \mathcal{H}_{\pi}$ we can write $g = \sum_{i = 1}^{\infty}c_i \pi(x_i)g$ for constants $c_i \in \mathbb{C}$ and elements $x_i \in G$. Then for $x \in G$ it follows that 
\begin{equation*}
    T(\pi(x)g) = T\left(\pi(x)\sum_{i = 1}^{\infty}c_{i}\pi(x_i)g\right) = \sum_{i = 1}^{\infty}c_{i}T\left(\pi(xx_i)g\right) = \sum_{i = 1}^{\infty}c_{i}\rho(xx_i)h = \rho(x)T(g). \qedhere
\end{equation*}
\end{proof}

Notice that Theorem \ref{equality_of_wavelet_spaces_general_statement} trivially implies that whenever $\pi$ and $\rho$ are not equivalent, then we necessarily have trivial intersection \[\mathcal{W}_{g}(\mathcal{H}_{\pi}) \cap \mathcal{W}_{h}(\mathcal{H}_{\rho}) = \{0\},\]
for any admissible vectors $g \in \mathcal{H}_{\pi}$ and $h \in \mathcal{H}_{\rho}$. The first application of Theorem \ref{equality_of_wavelet_spaces_general_statement} is a new proof of the result \cite[Theorem 4.2]{cont_wavelet_transform} which we state in Corollary \ref{equality_of_wavelet_spaces_special_case} below. This was originally proved by utilizing the orthogonality relations \eqref{wavelet_transform_orthogonality} for the wavelet transform. Recently, the result has been re-proven in the Garbor case in \cite[Lemma 3.3]{luef2020wiener} with the use of quantum harmonic analysis. For us, the result follows immediately from Theorem \ref{equality_of_wavelet_spaces_general_statement} together with Schur's Lemma \ref{Schur's_lemma}.

\begin{corollary}
\label{equality_of_wavelet_spaces_special_case}
Let $\pi:G \to \mathcal{U}(\mathcal{H}_{\pi})$ be a square integrable representation with admissible vectors $g, h \in \mathcal{H}_{\pi}$. If $\mathcal{W}_{h}(\mathcal{H}_{\pi}) \cap \mathcal{W}_{g}(\mathcal{H}_{\pi}) \neq \{0\}$ then $\mathcal{W}_{h}(\mathcal{H}_{\pi}) = \mathcal{W}_{g}(\mathcal{H}_{\pi})$ and $h = cg$ for some $c \in \mathbb{T}$.
\end{corollary}

\begin{remark}
The orthogonality relations \eqref{wavelet_transform_orthogonality} shows that the wavelet spaces $\mathcal{W}_{g}(\mathcal{H}_{\pi})$ and $\mathcal{W}_{h}(\mathcal{H}_{\pi})$ are orthogonal if and only if \[\langle C_{\pi}g, C_{\pi}h \rangle = 0,\] where $C_{\pi}$ is the Duflo-Moore operator. When $\langle C_{\pi}g, C_{\pi}h \rangle \neq 0$ the wavelet spaces still intersect trivially by Corollary \ref{equality_of_wavelet_spaces_special_case} except in the case $h = cg$ with $c \in \mathbb{T}$. 
\end{remark}

Before moving on, we show how we can combine Corollary \ref{equality_of_wavelet_spaces_special_case} with abstract results regarding functions of positive type to deduce concrete results for the wavelet transform.

\begin{definition}
A function $f:G \to \mathbb{C}$ on a locally compact group $G$ is said to be a function of \textit{(strictly) positive type} if for any finite subset $\Omega := \{x_1, \dots, x_m\} \subset G$, the matrix
\[\left\{f\left(x_{j}^{-1}x_{i}\right)\right\}_{i,j = 1}^{m}\]
is (strictly positive definite) positive semi-definite. 
\end{definition}

Let $\pi:G \to \mathcal{U}(\mathcal{H}_{\pi})$ be a square integrable representation with an admissible vector $g \in \mathcal{H}_{\pi}$. Then $\mathcal{W}_{g}g$ is a function of positive type due to Proposition \ref{spaces_are_RKHS} and the equality 
\begin{equation}
\label{proper_translation_invariance}
    \mathcal{W}_{g}g(x_{j}^{-1}x_i) = L_{x_{j}}\mathcal{W}_{g}g(x_i) = \mathcal{W}_{g}(\pi(x_j)g)(x_i), \qquad x_i,x_j \in G.
\end{equation}

\begin{corollary}
\label{original_consequence_1}
Let $\pi:G \to \mathcal{U}(\mathcal{H}_{\pi})$ and $\rho:G \to \mathcal{U}(\mathcal{H}_{\rho})$ be square integrable representations with admissible vectors $g \in \mathcal{H}_{\pi}$ and $h \in \mathcal{H}_{\rho}$, respectively. Then $\mathcal{W}_{g}g - \mathcal{W}_{h}h$ is never a non-zero function of positive type.
\end{corollary}

\begin{proof}
Assume that $\mathcal{W}_{g}g - \mathcal{W}_{h}h$ is a function of positive type. Then Aronszajn’s inclusion theorem \cite[Theorem 1.7.1]{aronszajn1950theory} in reproducing kernel Hilbert space theory implies that $\mathcal{W}_{h}(\mathcal{H}_{\rho}) \subset \mathcal{W}_{g}(\mathcal{H}_{\pi})$. Hence Theorem \ref{equality_of_wavelet_spaces_general_statement} implies that $\mathcal{W}_{h}(\mathcal{H}_{\pi}) = \mathcal{W}_{g}(\mathcal{H}_{\pi})$ and that $h = T(g)$ for some unitary intertwining operator $T:\mathcal{H}_{\pi} \to \mathcal{H}_{\rho}$. This implies for $x \in G$ that \begin{align*}
    \mathcal{W}_{g}g(x) - \mathcal{W}_{h}h(x) & = \langle g, \pi(x)g \rangle - \langle T(g), \rho(x)T(g) \rangle \\ & = \langle g, \pi(x)g \rangle - \langle T(g), T(\pi(x)g) \rangle \\ & = \langle g, \pi(x)g \rangle - \langle g, \pi(x)g \rangle \\ & = 0. \qedhere
\end{align*}
\end{proof}

For a locally compact group $G$ we let $\mathcal{P}_{c}$ denote the functions $f:G \to \mathbb{C}$ of positive type such that $f(e) = c \in \mathbb{C}$, where $e$ is the identity element of $G$. 

\begin{corollary}
\label{convex_combination_corollary}
Let $\pi:G \to \mathcal{U}(\mathcal{H}_{\pi})$ be a square integrable representation of a unimodular group $G$ with admissible vectors $g,g_1,g_2 \in \mathcal{H}_{\pi}$. Assume we can write $\mathcal{W}_{g}g$ as a convex combination \[\mathcal{W}_{g}g = t \cdot \mathcal{W}_{g_1}g_{1} + (1 - t) \cdot \mathcal{W}_{g_2}{g_2},\] for some $t \in [0,1]$. Then $t \in \{0,1\}$ and we either have $g = cg_1$ or $g = c g_2$ for some $c \in \mathbb{T}$.
\end{corollary}

\begin{proof}
Notice that $\mathcal{W}_{g}g \in \mathcal{P}_{c_{\pi}^{-1}}(G)$ where $C_{\pi} = c_{\pi} \cdot Id_{\mathcal{H}_{\pi}}$ since  \[\mathcal{W}_{g}g(e) = \|g\|_{\mathcal{H}_{\pi}} = c_{\pi}^{-1}.\] It follows from \cite[Theorem C.5.2]{bekka2008kazhdan} that the functions $\mathcal{W}_{g}g$ are extreme points in the bounded convex set $P_{c_{\pi}^{-1}}(G)$. This implies that $t = 0$ or $t=1$ and hence $\mathcal{W}_{g}g = \mathcal{W}_{g_1}g_1$ or $\mathcal{W}_{g}g = \mathcal{W}_{g_2}g_2$. Although what is said so far is well known, we can now apply Corollary \ref{equality_of_wavelet_spaces_special_case} to conclude that $g = cg_1$ or $g = c g_2$ for some $c \in \mathbb{T}$.
\end{proof}

\begin{example}
Let us check that everything works out for the STFT. Assume for normalized vectors $g,g_1,g_2 \in L^{2}(\mathbb{R}^{n})$ that \begin{equation}
\label{STFT_convexity}
    V_{g}g(x,\omega) = t \cdot V_{g_1}g_1(x,\omega) + (1 - t) \cdot V_{g_2}g_2(x,\omega),
\end{equation}
for some $t \in [0,1]$ and for all $(x,\omega) \in \mathbb{R}^{2n}$. Then by multiplying with $e^{-2\pi i \tau}e^{\pi i x \cdot \omega}$ on both sides we obtain \[\mathcal{W}_{g}g\Big(x,\omega,e^{2\pi i \tau}\Big) = t \cdot \mathcal{W}_{g_1}g_1\Big(x,\omega,e^{2\pi i \tau}\Big) + (1 - t) \cdot \mathcal{W}_{g_2}g_2\Big(x,\omega,e^{2\pi i \tau}\Big), \quad \Big(x,\omega,e^{2\pi i \tau}\Big) \in \mathbb{H}_{r}^{n}.\] We can now apply Corollary \ref{convex_combination_corollary} to see that $g = cg_1$ or $g = c g_2$ for some $c \in \mathbb{T}$. \par
In this specialized setting we describe an alternative proof using quantum mechanical reasoning. Assume again that \eqref{STFT_convexity} holds for some $t \in [0,1]$. For $g \in L^{2}(\mathbb{R}^{n})$ the \textit{Wigner distribution} $Wg$ in quantum mechanics can be defined through the STFT by the formula \[Wg(x,\omega) := 2^n e^{4 \pi i x \cdot \omega}V_{\mathcal{I}(g)}g(2x,2\omega), \qquad \mathcal{I}(g)(x) := g(-x).\] Hence \eqref{STFT_convexity} is equivalent to 
\[Wg(x,\omega) = t \cdot Wg_1(x,\omega) + (1 - t) \cdot Wg_2(x,\omega),\]
for all $(x,\omega) \in \mathbb{R}^{2n}$. One can now use the Weyl-quantization to go between functions on $\mathbb{R}^{2n}$ and operators on $L^{2}(\mathbb{R}^n)$. In this correspondence the Wigner distributions $Wg$ for $g \in L^{2}(\mathbb{R}^{n})$ correspond to the positive rank-one operators $g \otimes g$ given by \[(g \otimes g)(f) := \langle f, g \rangle \cdot g, \qquad f \in L^{2}(\mathbb{R}^{n}).\] Hence we obtain 
\[g \otimes g = t \cdot g_1 \otimes g_1 + (1 - t) \cdot g_2 \otimes g_2.\]
One can easily see by evaluation that this forces the same conclusion, namely that $g = cg_1$ or $g = c g_2$ for some $c \in \mathbb{T}$.
\end{example}

\section{Interpolation in Wavelet Spaces}
\label{sec: Interpolation_Problem}

We have seen on multiple occasions that the reproducing kernel Hilbert space structure of the wavelet spaces is immensely useful. We now focus in on that structure by considering a non-trivial interpolation problem. In this section we will describe the interpolation problem and show that the answer is not always affirmative. As we will see in the Section \ref{sec: Relationship_With_the_HRT-Conjecture}, the interpolation problem turns out to be equivalent to the HRT-Conjecture for the Gabor spaces. 

\begin{definition}
Let $X$ be a set and consider the distinct points $\Omega := \{x_1, \dots, x_m\} \subset X$ and possibly non-distinct scalars $\lambda_1, \dots, \lambda_m \in \mathbb{C}$. We say that a function $F:X \to \mathbb{C}$ \textit{interpolates} these points whenever $F(x_i) = \lambda_i$ for all $i = 1, \dots, m$. The function $F$ is called an \textit{interpolating function}.
\end{definition}

The question in interpolation theory is whether we can find an interpolating function with additional requirements. Typically, we have a Hilbert space $\mathcal{H}$ of functions on $X$ and ask whether we can choose $F \in \mathcal{H}$ as an interpolating function. When $\mathcal{H}$ is a reproducing Hilbert space, we can give an explicit criterion through the reproducing kernel. We state this result for the case we have investigated. 

\begin{proposition}
\label{sampeling_result}
Let $\pi:G \to \mathcal{U}(\mathcal{H}_{\pi})$ be a square integrable representation and fix an admissible vector $g \in \mathcal{H}_{\pi}$. Consider distinct points $\Omega := \{x_1, \dots, x_m\} \in G$ and possibly non-distinct scalars $\lambda_1, \dots, \lambda_m \in \mathbb{C}$. There exists an interpolating function $F \in \mathcal{W}_{g}(\mathcal{H}_{\pi})$ if and only if the vector $(\lambda_1, \dots, \lambda_m)^{T} \in \mathbb{C}^m$ is in the image of the $m \times m$ matrix 
\begin{equation*}
    K _{\Omega} :=  \left\{K(x_i,x_j)\right\}_{i,j = 1}^{m},
\end{equation*}
where $K$ is the reproducing kernel for the wavelet space $\mathcal{W}_{g}(\mathcal{H}_{\pi})$.
\end{proposition}

The proof of Proposition \ref{sampeling_result} follows from Proposition \ref{spaces_are_RKHS} together with \cite[Theorem 3.4]{paulsen2016introduction}. We remarked in Section \ref{sec: Reproducing_Kernel_Hilbert_Spaces} that the matrices $K_{\Omega}$ are always positive semi-definite. The interpolation problem in Proposition \ref{sampeling_result} have a unique solution for all $\Omega = \{x_1, \dots, x_m\} \subset G$ and $\lambda_1, \dots, \lambda_m \in \mathbb{C}$ if and only if the matrices $K_{\Omega}$ are all strictly positive definite. This is the case if and only if the function $\mathcal{W}_{g}g$ is a function of strictly positive type. This is the motivation for the terminology \textit{fully interpolating} given in Section \ref{sec: Reproducing_Kernel_Hilbert_Spaces}. Notice that for the point kernels $k_{x_1}, \dots, k_{x_m}$ we can write
\begin{equation*}
\sum_{i,j = 1}^{m} \overline{\alpha_i}\alpha_j k_{x_j}(x_i) = \left\langle \sum_{j = 1}^{m} \alpha_j k_{x_j}, \sum_{i = 1}^{m} \alpha_i k_{x_i} \right\rangle = \left\|\sum_{i = 1}^{m}\alpha_i k_{x_i}\right\|^2 \geq 0,  
\end{equation*}
for $\alpha_1, \dots, \alpha_m \in \mathbb{C}$. Hence $\mathcal{W}_{g}(\mathcal{H}_{\pi})$ is fully interpolating precisely when there are no non-trivial linear combinations between the point kernels $k_{x_1}, \dots, k_{x_m}$ for any points $x_1, \dots, x_m \in G$. 

\begin{remark}
It is straightforward to check that Proposition \ref{sampeling_result} is also valid for the Gabor spaces $V_{g}(L^{2}(\mathbb{R}^{n}))$. In that case, the point kernel corresponding to $(x,\omega) \in \mathbb{R}^{2n}$ is $k_{(x,\omega)} = V_{g}(M_{\omega}T_{x}g)$. Notice however that we get the extra phase-factor 
\begin{equation}
\label{extra_phase_factors}
    V_{g}(M_{\omega}T_{x}g)(s,t) = e^{-2 \pi i x \cdot (t - \omega)}V_{g}g(s - x, t - \omega), \quad (s,t) \in \mathbb{R}^{2n},
\end{equation}
in contrast with \eqref{proper_translation_invariance}. 
\end{remark}

When $G = \{e\}$ the only wavelet space associated with $G$ is the one-dimensional space $L^{2}(G)$. This is fully interpolating for trivial reasons. We exclude this case in future examples and refer to a locally compact group $G$ as \textit{non-trivial} when $G$ has more than one element. The next result shows that a large class of wavelet spaces are not fully interpolating. 

\begin{proposition}
\label{no_abelian_or_compact_are_fully_interpolating}
Let $G$ be a non-trivial locally compact group. If $G$ is either abelian or compact then no wavelet space associated to $G$ is fully interpolating. 
\end{proposition}

\begin{proof}
An abelian locally compact group $G$ possesses square integrable representations if and only if the group is compact. In this case, the representation theory of compact groups shows that any irreducible unitary representation of $G$ is finite-dimensional \cite[Theorem 5.2]{folland2016course}. Any irreducible unitary representation of $G$ is also automatically square integrable due to the compactness of $G$. If $G$ is an infinite group, then we can always pick $\Omega = \{x_1, \dots, x_m \} \subset G$ to have larger cardinality than the dimension of the representation considered. Then there is no way that $k_{x_1}, \dots, k_{x_m}$ can be linearly independent. If $G$ is a finite group, then the same argument goes through unless $G$ have an irreducible representation whose dimension is greater or equal to the order of the group $G$. This is not possible since the class equation in finite representation theory gives that 
\begin{equation*}
    |G| = \sum_{[\pi]}\textrm{dim}(\mathcal{H}_{\pi})^2,
\end{equation*}
where the sum runs over all equivalence classes of irreducible representation $\pi$ of $G$. Since we have excluded $G$ from being the trivial group, the result follows.
\end{proof}

\begin{example}
\label{example_STFT_of_Gaussian}
For the $n$-dimensional Gaussian function $g_{n}(x) := e^{-\frac{\pi}{2} x^2}$ we will show that the Gabor space $V_{g_{n}}(L^{2}(\mathbb{R}^{n}))$ is fully interpolating. A straightforward computation reveals that \[V_{g_n}g_n(x,\omega) = e^{- \pi i x \cdot \omega} e^{-\frac{\pi}{4}x^2}e^{-\pi \omega^2}, \qquad (x,\omega) \in \mathbb{R}^{2n}.\]
Assume by contradiction that there is a linear dependence between the point kernels $k_{(x_k,\omega_k)}$ corresponding to distinct points $(x_k,\omega_k) \in \mathbb{R}^{2n}$ for $k = 1, \dots, m$. The linear dependence explicitly gives 
\begin{equation*}
    \sum_{k = 1}^{m} \alpha_k e^{2 \pi i x_k \cdot \omega_k} e^{-2 \pi i x_k \cdot \omega}e^{- \pi i (x - x_k)\cdot(\omega - \omega_k)}e^{-\frac{\pi}{4}(x - x_k)^{2}}e^{-\pi(\omega - \omega_k)^{2}} = 0,
\end{equation*}
where $\alpha_1, \dots, \alpha_n \in \mathbb{C}$ are not all zero. By setting $\beta_k = \alpha_k e^{\pi i x_k \cdot \omega_k}e^{-\frac{\pi}{4}x_k^2}e^{-\pi \omega_{k}^2}$ we obtain 
\begin{equation*}
    e^{-\pi i x \cdot \omega}e^{-\frac{\pi}{4}x^2} e^{-\pi \omega^2}\sum_{k = 1}^{m}\beta_k e^{-2 \pi i x_k \cdot \omega}e^{\pi i(x\cdot\omega_k + \omega \cdot x_k)}e^{\frac{\pi}{2}x \cdot x_k}e^{2 \pi \omega \cdot \omega_k} = 0.
\end{equation*}
We can divide by the non-zero function $e^{-\pi i x \cdot \omega}e^{-\frac{\pi}{4}x^2} e^{-\pi \omega^2}$ and set $\omega = 0$ to get the simplified equation
\begin{equation}
\label{last_sum_equation}
    \sum_{k = 1}^{m}\beta_k e^{x \cdot \left(\frac{\pi}{2}x_k + i \omega_k\right)} = 0.
\end{equation}
Notice that the coefficients $\beta_k$ satisfy $\beta_k = 0$ if and only if $\alpha_k = 0$. The equation \eqref{last_sum_equation} contradicts the independence of the exponential functions $x \mapsto e^{x \cdot \lambda_k}$, see e.g. \cite[Lemma 13.1]{cheney2009course}, since $\lambda_k = \frac{\pi}{2}x_k + i \omega_k$ are distinct complex numbers. 
\end{example}

\begin{example}
\label{example_with_plot}
To illustrate that the Gabor space $V_{g_{1}}(L^{2}(\mathbb{R})) \subset L^{2}(\mathbb{R}^{2})$ is fully interpolating we consider the points $x_1 = (0,0)$, $x_2 = (1,0),$ and $x_3 = (0,1)$ in $\mathbb{R}^{2}$ along with $\lambda_1 = \lambda_2 = \lambda_3 = 1$. Then there exists $F \in V_{g_{1}}(L^{2}(\mathbb{R}))$ such that $F(x_i) = \lambda_i$ for $i = 1,2,3$. Moreover, the function $F \in V_{g_{1}}(L^{2}(\mathbb{R}))$ with minimal norm that interpolates these points will be on the form \[F(x,\omega) = \alpha_1 V_{g_1}g_1(x,\omega) + \alpha_2 V_{g_1}g_1(x - 1,\omega) + \alpha_2 V_{g_1}g_1(x,\omega - 1),\]
for some $\alpha_1,\alpha_2,\alpha_3 \in \mathbb{C}$ by \cite[Theorem 3.4]{paulsen2016introduction}.
It follows by straightforward computations that $\alpha_1 \simeq 0.6218$, $\alpha_2 \simeq 0.7360$, and $\alpha_3 \simeq 0.9876$.
\begin{figure}[h!]
\centering
\includegraphics[scale=0.17]{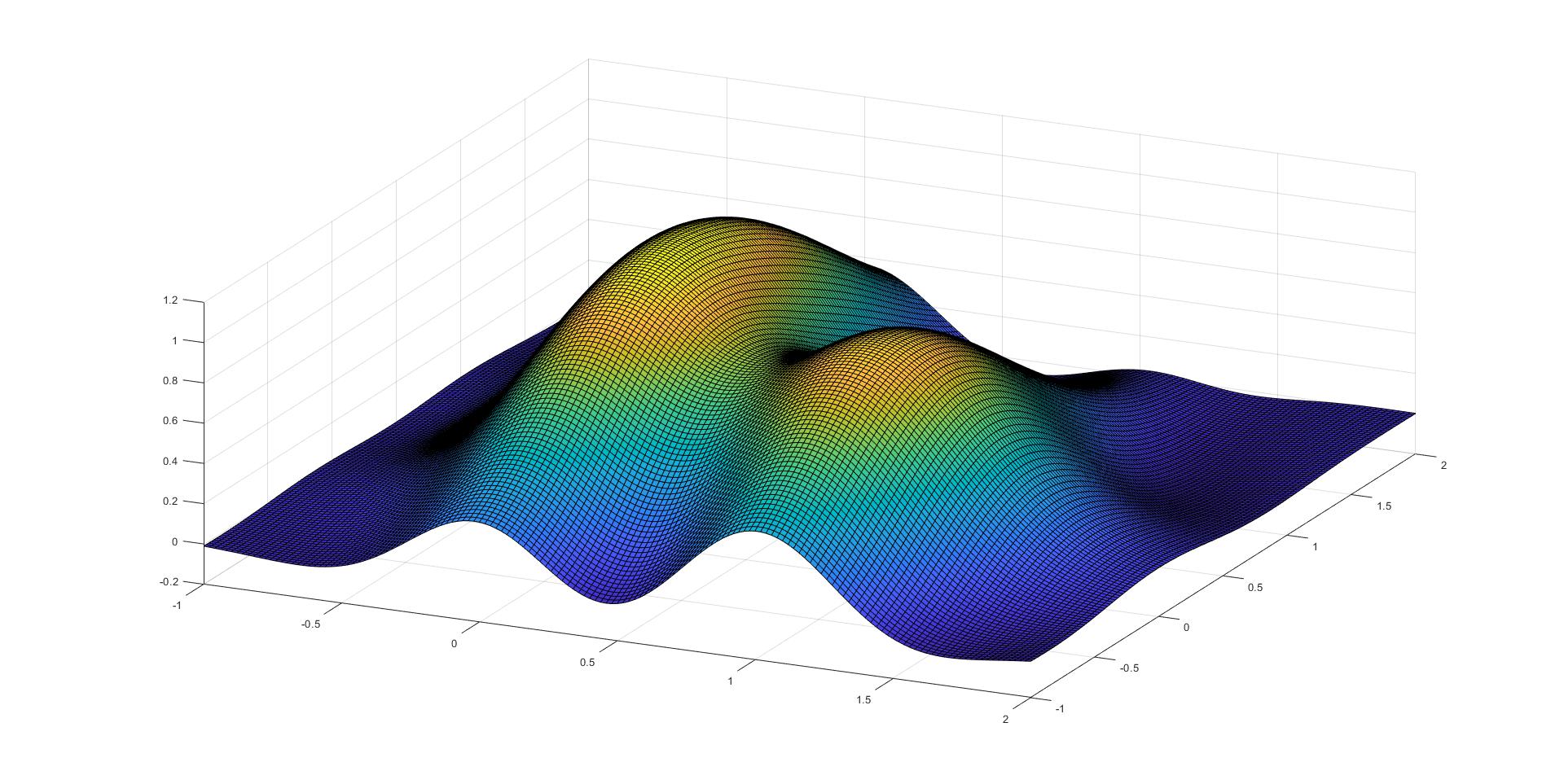}
\caption{The real part of the interpolating function $F(x,\omega)$ in Example \ref{example_with_plot}.}
\end{figure}
\end{example}

\begin{remark}
The function $V_{g_n}g_n$ on $\mathbb{R}^{2n}$ is not strictly positive definite even though the Gabor space $V_{g_n}(L^{2}(\mathbb{R}^{n}))$ is fully interpolating. This discrepancy is due to the extra phase-factor in \eqref{extra_phase_factors}. In fact, the function $V_{g}g(x,\omega)$ is not even positive definite: If this was the case, then the Fourier inverse $\mathcal{F}^{-1}(V_{g}g)$ would be a positive function on $\mathbb{R}^{2n}$ by Bochner's Theorem \cite[Theorem 10.4]{paulsen2016introduction}. However, the function $\mathcal{F}^{-1}(V_{g}g)(x,\omega) = e^{-\pi(x^2 + \omega^2)} e^{2 \pi i x \cdot \omega}$ is clearly not even real valued. 
\end{remark}

\section{Connection With the HRT-Conjecture}
\label{sec: Relationship_With_the_HRT-Conjecture}

The question of whether the Gabor spaces are fully interpolating turns out to be equivalent to the infamous HRT-Conjecture. Recall that a subset $\mathcal{A}\subset \mathcal{H}$ of a vector space $\mathcal{H}$ is said to be \textit{linearly independent} if every finite subset $F \subset \mathcal{A}$ is linearly independent in the classical sense. The following open conjecture reveals how little is understood about time-frequency shifts.

\begin{conjecture*}[HRT]
Is the set \[\left\{M_{\omega}T_{x}g\right\}_{(x,\omega) \in \mathbb{R}^{2n}}\] linearly independent in $L^{2}(\mathbb{R}^{n})$ for all non-zero $g \in L^{2}(\mathbb{R}^{n})$?
\end{conjecture*}

The HRT-Conjecture was originally posed back in 1996 by C. Heil, J. Ramanathan, and P. Topiwala in the paper \cite{heil1996linear}. There have been many significant developments on the conjecture during the years, where techniques from von-Naumann algebras \cite{linnell1999neumann}, spectral theory \cite{balan2010almost}, ergodic theory \cite{heil2006linear}, and representation theory of the Heisenberg groups \cite{currey_oussa} have been used. We refer the reader to the introduction of the paper \cite{okoudjou2019extension} for a reasonably extensive list of contributions to the HRT conjecture. Moreover, we recommend the survey papers \cite{heil2006linear, heil2015hrt} on the HRT-Conjecture written by one of its founders. The following result shows that the HRT-Conjecture can be reformulated to a problem regarding reproducing kernel Hilbert spaces. 

\begin{proposition}
\label{equvalence_between_fully_interpolating_and_HRT_conjecture}
The HRT-Conjecture is equivalent to the statement that the Gabor spaces are fully interpolating.
\end{proposition}

\begin{proof}
Let us fix elements $(x_1,\omega_1), \dots, (x_m,\omega_m) \in \mathbb{R}^{2n}$ and consider the collection \begin{equation}
\label{time_frequency_shifts_independence}
    \{M_{\omega_k}T_{x_k}g\}_{k = 1}^{m}.
\end{equation} 
We henceforth assume that $\|g\|_{L^{2}(\mathbb{R}^{n})} = 1$ since normalizing $g \in L^{2}(\mathbb{R}^{n})$ does not change whether the collection \eqref{time_frequency_shifts_independence} is linearly independent. \par 
Assume first that the collection \eqref{time_frequency_shifts_independence} is linearly dependent, that is, there exist $\alpha_1, \dots, \alpha_m \in \mathbb{C}$ not all zero such that \[\sum_{k = 1}^{m} \alpha_{k}M_{\omega_k}T_{x_k}g = 0.\]
We can take the inner-product with the function $M_{\omega}T_{x}g$ to obtain 
\begin{equation*}
    \sum_{k = 1}^{m} \alpha_{k}\left\langle M_{\omega_k}T_{x_k}g, M_{\omega}T_{x}g \right \rangle = \sum_{k = 1}^{m} \alpha_{k} V_{g}(M_{\omega_k}T_{x_k}g)(x,\omega) = \sum_{k = 1}^{m} \alpha_{k} k_{(x_k,\omega_k)}(x,\omega) = 0.
\end{equation*}
This gives a linear dependence between $k_{(x_1,\omega_1)}, \dots, k_{(x_m,\omega_m)}$, showing that $V_{g}(L^{2}(\mathbb{R}^{n}))$ is not fully interpolating. \par 
Conversely, assume that $V_{g}(L^{2}(\mathbb{R}^{n}))$ is not fully interpolating. Then there exists a linear dependence between the point kernels $k_{(x_1,\omega_1)}, \dots, k_{(x_m,\omega_m)}$ for some points $(x_1, \omega_1), \dots, (x_m,\omega_m) \in \mathbb{R}^{2n}$. Retracing the steps we took previously we conclude that \[\sum_{k = 1}^{m} \alpha_{k}\left\langle M_{\omega_k}T_{x_k}g, M_{\omega}T_{x}g \right \rangle = 0,\]
where $\alpha_1, \dots, \alpha_m \in \mathbb{C}$ are not all zero. The proof of Proposition \ref{completeness_proposition} shows that the collection $\{M_{\omega}T_{x}g\}_{(x,\omega) \in \mathbb{R}^{2n}}$ is complete in $L^{2}(\mathbb{R}^{n})$. This implies the linear dependence \[\sum_{k = 1}^{m}\alpha_k M_{\omega_k}T_{x_k}g = 0. \qedhere\]
\end{proof}

Proposition \ref{equvalence_between_fully_interpolating_and_HRT_conjecture} allows us to use the partial results available on the HRT-Conjecture in the literature to deduce that certain Gabor spaces are fully interpolating. In particular, it was known from the beginning \cite[Proposition 4]{heil1996linear} that the HRT-Conjecture is true for the $n$-dimensional Gaussian function. In Example \ref{example_STFT_of_Gaussian} we proved, in light of Proposition \ref{equvalence_between_fully_interpolating_and_HRT_conjecture}, the same thing by brute-force calculations with the short-time Fourier transform. We can use \cite[Proposition 4]{heil1996linear} and Proposition \ref{equvalence_between_fully_interpolating_and_HRT_conjecture} to conclude that the Gabor spaces $V_{g}(L^{2}(\mathbb{R}^{n}))$ are fully interpolating whenever $g$ is a Hermite function.

\vspace{0.25cm}
\begin{remark}
A careful read of the proof of Proposition \ref{equvalence_between_fully_interpolating_and_HRT_conjecture} reveals that the statement is true in the generalized setting. More precisely, let $\pi:G \to \mathcal{U}(\mathcal{H}_{\pi})$ be a square integrable representation with an admissible vector $g \in \mathcal{H}_{\pi}$. Then the collection $\{\pi(x)g\}_{x \in G}$ is linearly independent in $\mathcal{H}_{\pi}$ if and only if the wavelet space $\mathcal{W}_{g}(\mathcal{H}_{\pi}) \subset L^{2}(G)$ is fully interpolating. Hence the problem of whether the wavelet space $\mathcal{W}_{g}(\mathcal{H}_{\pi})$ is fully interpolating is a convenient generalization of the HRT-Conjecture.
In this reformulation, Proposition \ref{no_abelian_or_compact_are_fully_interpolating} states that the generalized HRT-Conjecture is false for compact or abelian groups. Moreover, the generalized HRT-Conjecture is also false in the classical wavelet setting \cite{heil2015hrt} as a result of the scaling relation in wavelet theory. Another generalization of the HRT-Conjecture is considered in \cite{kutyniok2002linear}.   
\end{remark}
\vspace{0.25cm}

Recently there has been an effort to prove the HRT-Conjecture for widely spaced index sets \cite{kreisel2019letter, nicola2019note}. In particular, it is showed in \cite[Theorem 1]{kreisel2019letter} that the HRT-Conjecture holds for $g \in C_{0}(\mathbb{R}^{n})$ and points $\Omega := \{(x_1,\omega_1), \dots, (x_m,\omega_m))\} \subset \mathbb{R}^{2n}$ that are widely spaced apart relative to the decay of $g$. Through our approach, we can deduce a similar result without the assumption that $g \in C_{0}(\mathbb{R}^{n})$ since the STFT satisfies $V_{g}g \in C_{0}(\mathbb{R}^{2n})$ for all $g \in L^{2}(\mathbb{R}^{n})$.

\begin{corollary}
Let $g \in L^{2}(\mathbb{R}^n)$ be a non-zero function. There exists $R > 0$ (depending only on $g$ and $m \in \mathbb{N}$) such that for any collection of points $(x_1,\omega_1), \dots, (x_m, \omega_m) \in \mathbb{R}^{2n}$ with 
\begin{equation}
\label{far_spread_points}
    \min_{i \neq j}\sqrt{(x_j - x_i)^2 + (\omega_j - \omega_i)^2} \geq R, \qquad i,j = 1, \dots, m,
\end{equation}
the time-frequency shifts $\{M_{\omega_k}T_{x_k}g\}_{k = 1}^{m}$ are linearly independent. 
\end{corollary}

\begin{proof}
We assume that $\|g\|_{L^{2}(\mathbb{R}^n)} = 1$ as we can normalize $g$ without altering the linear independence. The claim is equivalent, by Proposition \ref{equvalence_between_fully_interpolating_and_HRT_conjecture}, to the fact that the matrix \[\Omega_{g} := \big\{\big\langle V_{g}(M_{\omega_{j}}T_{x_{j}}g), V_{g}(M_{\omega_{i}}T_{x_{i}}g) \big\rangle \big\}_{i,j = 1}^{m} = \big\{e^{-2\pi i x_j \cdot  (\omega_i - \omega_j)}V_{g}g(x_i - x_j, \omega_i - \omega_j)\big\}_{i,j = 1}^m\] is invertible. Notice that the diagonal terms of $\Omega_{g}$ are all $1$'s. Since $V_{g}g$ is continuous and vanishes at infinity, we can find $R > 0$ such that \[\sum_{j = 1}^{m}|V_{g}g(x_i - x_j, \omega_i - \omega_j)| \leq 1, \qquad i = 1, \dots, m,\] for all points $(x_1,\omega_1), \dots, (x_m, \omega_m)$ satisfying the condition \eqref{far_spread_points}. This guarantees that the matrix $\Omega_{g}$ is diagonally dominant and hence invertible. 
\end{proof}

\begin{remark}
We would like to bring up that Proposition \ref{equvalence_between_fully_interpolating_and_HRT_conjecture} is implicitly commented on in the paper \cite{grochenig2015linear} through frame theory terminology. More precisely, the author investigates the \textit{Grammian matrix} corresponding to the time-frequency shifts $\{M_{\omega_k}T_{x_k}g\}_{k = 1}^{m}$. The invertibility of the Grammian matrix is easily seen to be equivalent to the statement that the corresponding Gabor space $V_{g}(L^{2}(\mathbb{R}^{n}))$ is fully interpolating. We hope the connection with reproducing kernel Hilbert spaces adds a machinery that can help shed light on some aspects of the HRT-Conjecture. 
\end{remark}

\section{Wavelet Completeness}
\label{sec: Wavelet_Completeness}
In this final section we will look at how much of $L^{2}(G)$ the wavelet spaces $\mathcal{W}_{g}(\mathcal{H}_{\pi})$ collectively fill up. Let $\pi:G \to \mathcal{U}(\mathcal{H}_{\pi})$ be a square integrable representation and let $\mathcal{A}_{\pi}$ denote the equivalence classes of admissible vectors in $\mathcal{H}_{\pi}$ modulo rotations by elements of $\mathbb{T}$. From Example \ref{Example_abelian_groups} we see that the collection
\begin{equation*}
        \underset{g \in \mathcal{A}_{\pi}}{\textrm{span}}\big\{\mathcal{W}_{g}f \, : \, f \in \mathcal{H}_{\pi}\big\} \subset L^{2}(G)
\end{equation*} 
does not need to be dense in $L^{2}(G)$. To combat this we will start to vary the square integrable representation $\pi$ as well. If $\widehat{G}_{s}$ denotes the equivalence classes of square integrable representations of $G$, then we consider
\begin{equation}
\label{direct_sum_before_closure}
        \bigoplus_{\pi \in \widehat{G}_s}
        \underset{g \in \mathcal{A}_{\pi}}{\textrm{span}}\big\{\mathcal{W}_{g}f \, : \, f \in \mathcal{H}_{\pi}\big\} \subset L^{2}(G).
\end{equation} 
It is straightforward to check that \eqref{direct_sum_before_closure} is a well-defined direct sum, see \cite[Lemma 2.24]{fuhr2005abstract} for details.

\begin{example}
To make matters more concrete, let us first consider the group $G = \mathbb{T}$. Any unitary representation of $\mathbb{T}$ is equivalent through a unitary intertwining operator to one of the representations $\pi_{n}:\mathbb{T} \to \mathbb{T}$ for $n \in \mathbb{Z}$ given by \[\pi_{n}\big(e^{i \theta}\big) := e^{i n \theta}, \quad \theta \in \mathbb{R}.\] For the representation $\pi_{n}$ we see that \[\mathcal{W}_{1}1\big(e^{i \theta}\big) = \big\langle 1, \pi_{n}\big(e^{i \theta}\big)1 \big\rangle = e^{-in\theta}.\] This gives precisely the Fourier expansion of square integrable periodic functions since
\[\overline{\bigoplus_{\pi \in \widehat{G}_s}
        \underset{g \in \mathcal{A}_{\pi}}{\textrm{span}}\big\{\mathcal{W}_{g}f \, : \, f \in \mathcal{H}_{\pi}\big\}} = \overline{ \bigoplus_{n \in \mathbb{Z}} \textrm{span}\big\{e^{in\theta} \,: \, \theta \in \mathbb{R}\big\}} = L^{2}(\mathbb{T}).\]
\end{example}

Based on the observations above, we formulate the following conjecture.

\begin{conjecture*}[Wavelet Completeness]
Characterize the locally compact groups $G$ that satisfy
\begin{equation}
\label{wavelet_conjecture}
    \begin{split}
        \overline{\bigoplus_{\pi \in \widehat{G}_s}
        \underset{g \in \mathcal{A}_{\pi}}{\rm{span}}\big\{\mathcal{W}_{g}f \, : \, f \in \mathcal{H}_{\pi}\big\}} = L^{2}(G).
    \end{split}
\end{equation}
\end{conjecture*}

We say that a locally compact group $G$ is \textit{wavelet complete} if \eqref{wavelet_conjecture} holds for $G$. For wavelet complete groups we can view the decomposition \eqref{wavelet_conjecture} conceptually as a generalized multiresolution analysis. An obvious condition that needs to be satisfied for $G$ to be wavelet complete is $\widehat{G}_{s} \neq \emptyset$. Hence $\mathbb{Z}$ and any other abelian non-compact group is not wavelet complete. Any compact group is easily seen to be wavelet complete from Peter-Weyl theory, see e.g. \cite[Theorem 5.11]{folland2016course}. The following example illustrates that wavelet completeness is a non-trivial notion.

\begin{proposition}
\label{wavelet_completeness_of_HG}
The reduced Heisenberg groups $\mathbb{H}_{r}^{n}$ are not wavelet complete. 
\end{proposition}

\begin{proof}
A variant of the Stone-von Neumann Theorem \cite[Corollary 9.3.5]{grochenig2001foundations} implies that the only square integrable representations of $\mathbb{H}_{r}^{n}$ are the Schr\"{o}dinger representation $\rho_{r}$ given in \eqref{Schrodinger_representation} along with appropriate dilations \[\rho_{r,m}\Big(x,\omega,e^{2\pi i \tau}\Big) := e^{2\pi i m \tau}e^{\pi i m x \cdot \omega}T_{mx}M_{\omega},\]
for $m \in \mathbb{Z} \setminus \{0\}$. We will show that any $h \in L^{2}(\mathbb{H}_{r}^{n})$ on the form $h\big(x,\omega,e^{2\pi i \tau}\big) = h(x,\omega)$ is orthogonal to $\mathcal{W}_{g}f$ for all $f,g \in L^{2}(\mathbb{R}^{n})$ and all the representations $\rho_{r,m}$. We compute that \begin{align*}
    \langle h, \mathcal{W}_{g}f \rangle_{L^{2}(\mathbb{H}_{r}^{n})} & = \int_{\mathbb{H}_{r}^{n}}h(x,\omega)\overline{\mathcal{W}_{g}f(x,\omega,e^{2\pi i \tau})} \, dx \, d\omega \, d\tau \\ & = \int_{0}^{1}e^{2 \pi i m \tau} \, d\tau \int_{\mathbb{R}^{2n}}h(x,\omega)e^{-\pi i m x \cdot \omega}\overline{V_{g}f(mx,\omega)} \, dx \, d\omega \\ & = 0,
\end{align*}
since $m \in \mathbb{Z} \setminus \{0\}$.
\end{proof}


\bibliographystyle{abbrv}
\bibliography{main}

\Addresses

\end{document}